\documentclass[11pt,a4paper,twoside]{article}
\usepackage{latexsym,amsfonts,amsmath, amsthm, amssymb}
\usepackage{fullpage}
\usepackage{xcolor,bm}
\usepackage[utf8x]{inputenc}
\usepackage{array}
\usepackage{mathrsfs}

\usepackage{fourier}

\newcommand{\R}{\mathbb R}
\newcommand{\N}{\mathbb N}

\newcommand{\Pro}{\mathbb P}

\newcommand{\Var}{\mathrm{Var}}

\newcommand{\Cov}{\operatorname{Cov}}
\providecommand{\abs}[1]{\left|#1\right|}
\providecommand{\norm}[1]{\left\|#1\right\|}
\providecommand{\pr}[1]{\mathbb{P}\left[ #1\right]}
\providecommand{\prb}[1]{\mathbb{P}\bigl[ #1\bigr]}
\providecommand{\ee}[1]{\mathbb{E}\left[ #1\right]}

\DeclareMathOperator{\conv}{conv}

\def\limd{\xrightarrow[n\to\infty]{d}}

\newtheorem{thm}{Theorem}[section]
\newtheorem{cor}[thm]{Corollary}
\newtheorem{lemma}[thm]{Lemma}

\newtheorem{proposition}[thm]{Proposition}

\newtheorem{thmalpha}{Theorem}

{
\theoremstyle{definition}

\newtheorem{rmk}[thm]{Remark}

}

\allowdisplaybreaks
\begin{document}

\title{\bf Limit theorems for random points in a simplex}

\medskip

\author{ Anastas Baci, Zakhar Kabluchko, Joscha Prochno, Mathias Sonnleitner, and Christoph Th\"ale}



\date{}

\maketitle

\begin{abstract}
\small
In this work the $\ell_q$-norms of points chosen uniformly at random in a centered regular simplex in high dimensions are studied. Berry-Esseen bounds in the regime $1\leq q < \infty$ are derived and complemented by a non-central limit theorem together with moderate and large deviations in the case where $q=\infty$. A comparison with corresponding results for $\ell_p^n$-balls is carried out as well.
\medspace
\end{abstract}


\section{Introduction and main results}

One of the central aspects of high-dimensional probability theory is the study of random geometric quantities and the phenomena that occur as the dimension of the ambient space tends to infinity. The field is intimately connected to geometric functional analysis as well as convex and discrete geometry, and has attracted considerable attention in the last decade. This is in parts because of numerous applications that can be found in the statistics and machine learning literature related to high-dimensional data, e.g., in form of dimensionality reduction in information retrieval \cite{BM2001, LG2003}, clustering \cite{FB2003, M2018}, principal component regression \cite{S2018}, community detection in networks \cite{FH2016, LLV2018}, topic discovery \cite{DRIS2013}, or covariance estimation \cite{CRZ2016, V2018}. A famous example for a high dimensional limit theorem is the Maxwell-Poincar\'e-Borel Lemma  (see, e.g., \cite{DF1987} or \cite[Lemma 1.2]{L1996}) stating that for fixed $k\in\N$, the distribution of the first $k$ coordinates of a point chosen uniformly at random from the $n$-dimensional Euclidean ball or sphere of radius one converges weakly to a $k$-dimensional Gaussian distribution as the space dimension $n$ tends to infinity.

Today, there is a vast literature on high-dimensional central limit theorems, which describe the Gaussian fluctuations for various random geometric quantities in different contexts, for instance, the famous central limit theorem for convex bodies \cite{K2007}, a central limit theorem for the volume of random projections of the cube \cite{PPZ14}, a central limit theorem for the Euclidean norm of random projections of points chosen randomly from $\ell_p^n$-balls \cite{APT2019} and several others, see \cite{AGetal2019,BV2007,GKT2019,GrygT20,GT20,JP2019,KPT2019_I,KPT2019_II,MM2007,R2005,Schmu2001,S82}. Other limit theorems, such as moderate deviations principles and large deviations principles, have only been studied in high-dimensional probability related to the geometry of convex bodies since their introduction by Kabluchko, Prochno, and Th\"ale \cite{KPT2019_II} and Gantert, Kim, and Ramanan \cite{GKR2017}. In fact, those kind of limit theorems are more sensitive to the randomness involved and display a non-universal behavior in their speed and/or rate function. The latter fact makes the subject particularly interesting as, contrary to a central limit theorem which implies the somewhat negative result that fluctuations do not provide much information because of universality, the moderate and large deviations limit theorems are distribution dependent and encode subtle geometric information about the underlying structure. In the past three years a number of interesting results in this direction have been obtained and we refer the reader to \cite{APT2018,GKR2016,KPT2019_cube,KPT2019_sanov,KR2018,KR2019,LR2020}. An interesting connection between the study of moderate and large deviations and the famous Kannan-Lov\'asz-Simonovits conjecture has recently been discovered by Alonso-Guti\'errez, Prochno, and Th\"ale in \cite{APT2020}.

In this work, we study limit theorems for (suitably normalized) $\ell_q$-norms of points chosen uniformly at random in a centered and regular simplex. When $1\leq q < \infty$, we provide a Berry-Esseen-type rate of convergence to a standard Gaussian distribution. For the case where $q=+\infty$, we complement this result with a non-central limit theorem, establishing the weak convergence to a Gumbel distribution, and provide both a moderate and large deviations principle. Let us point out that the method of proof in the Berry-Esseen-type central limit theorem differs from the previously mentioned ones in the sense that here we use a connection to the asymptotic theory of sums of random spacings (and not, e.g., a Schechtman-Zinn type probabilistic representation).

In order to be more precise, let $n\in\N$ and consider the $(n-1)$-dimensional simplex
\[
\Delta_{n-1}:=\Bigg\{x\in\R^n \,:\, x_i\geq 0,\, \sum_{i=1}^n x_i=1\Bigg\} = \conv\{e_1,\dots,e_n \} ,
\]
where $\conv(A)$ denotes the convex hull of a set $A$ and $e_1,\ldots,e_n$ stands for the unit vectors of the standard orthonormal basis of $\R^n$. Consider a sequence $(E_i)_{i\in\N}$ of independent random variables with an exponential distribution of mean 1 and, for each $n\in\N$, let $S_n:=\sum_{i=1}^n E_i$ denote the $n^{\text{th}}$ partial sum. We study the sequence of random vectors
\[
Z_n:=\left(\frac{E_1}{S_n}-\frac{1}{n},\dots,\frac{E_n}{S_n}-\frac{1}{n}\right)\in \Delta_{n-1}-\frac{1}{n}(e_1+\ldots+e_n),\qquad n\in\N.
\]
In fact, uniformly distributed points in the standard centered simplex in $\R^n$, $\Delta_{n-1}-\frac{1}{n}(e_1+\ldots+e_n)$, have the same distribution as $Z_n$ (see, e.g., \cite{MathaiBook}). A different method to generate uniform random vectors in the simplex is by letting $U_1,\ldots, U_{n-1}$ be independent and identically distributed random variables with a uniform distribution on $[0,1]$ and considering
\[
G_{n,i}:=U_{(i)}-U_{(i-1)}, \qquad i=1,\ldots,n,
\]
where $U_{(i)}$ is the $i^{\text{th}}$ order statistic of $U_1,\ldots,U_{n-1}$, with the convention that $U_{(0)}:=0, U_{(n)}:=1$. Then the vector $G_n:=(G_{n,1},\ldots,G_{n,n})$ is uniformly distributed in $\Delta_{n-1}$ and $Z_n\overset{d}{=} G_n-\frac{1}{n}(e_1+\ldots+e_n)$.
A proof of this fact can be found, for instance, in \cite[Chapter 6.4]{DN2003}.

The first main result of this paper is the following Berry-Esseen-type theorem for the $\ell_q$-norm $\|\,\cdot\,\|_q$ ($1\leq q<\infty$) of uniform random points in $\Delta_{n-1}-\frac 1n (e_1+\ldots+e_n)$. Define
\[
\mu_q:=\ee{\abs{E_1-1}^q}\quad \text{and} \quad \sigma_q^2:=q^{-2}\mu_q^{-2}(\mu_{2q}-(q^2+2q+2)\mu_q^2+2(q+1)\mu_q-1).
\]
For example $\mu_1=2e^{-1},\sigma_1=2e-5$, and $\mu_2=\sigma_2=1$.

\begin{thmalpha}\label{thm:berryesseen}
Let $1\leq q<\infty$. There exists some constant $c_q>0$ only depending on $q$ such that, for all $n\in \mathbb{N}$,
\[
\sup_{x\in\R}\left| \mathbb{P}\left[\sqrt{n}\left(\frac{n^{1-1/q}\Vert Z_n\Vert_q\,\mu_q^{-1/q} -1}{\sigma_q}\right)\leq x\right]-\mathbb{P}[N\leq x]\right| \leq c_q\frac{\log n}{\sqrt{n}},
\]
where $N\sim\mathcal N(0,1)$ is a standard Gaussian random variable.
\end{thmalpha}

In the proof of this result we use a connection to the asymptotic theory of sums of random spacings. This allows us to use a Berry-Esseen result of Mirakhmedov \cite{M2005}, who improved a theorem due to Does and Klaassen \cite{DK1984}.

As a direct corollary of Theorem \ref{thm:berryesseen}, we obtain the following central limit theorem, which is the analogue for the regular simplex of the corresponding central limit theorems in \cite{KPT2019_I,KPT2019_II} for $\ell_p^n$-balls.

\begin{cor}\label{cor:cltqnorm}
	For all $1\leq q<\infty$, we have
	\[
	\sqrt{n}\big(n^{1-1/q}\| Z_n \|_q\,\mu_q^{-1/q} -1\big)\limd Z \sim N(0,\sigma_q^2).
	\]
\end{cor}

When the parameter $q$ satisfies $q=\infty$, we cannot expect convergence in distribution of $\|Z_n\|_\infty$ to a Gaussian random variable. However, we establish a non-central limit theorem with a double exponential (also known as Gumbel) distribution in the limit. The result in this case reads as follows.

\begin{thmalpha}\label{thm:non-central limit theorem}
	We have
	\[
	n\Vert Z_n\Vert_{\infty}-(\log n-1)\limd G,
	\]
	where $G$ has a standard Gumbel distribution, i.e., $\mathbb{P}[G\leq x]=\exp(-e^{-x})$ for $x\in\R$.
\end{thmalpha}

The next result describes the upper and lower deviations on a moderate scale, which lies between the Gaussian fluctuations of a central limit theorem and the large deviations which occur on the scale of a law of large numbers. For a formal definition of a moderate deviations principle (MDP) and a large deviations principle (LDP) we refer to Section \ref{sec:LDPMDPprel} below.

\begin{thmalpha}\label{thm:moderate}
Let $(s_n)_{n\in\N}$ be a positive sequence with $s_n\to \infty$ and $s_n/\log n\to 0$. Then, the sequence $\left(\frac{\log n}{s_n}\left(\frac{n}{\log n}\|Z_n\|_{\infty}-1\right)\right)_{n\in\mathbb{N}}$ satisfies an MDP with speed $s_n$ and rate function
\[
\mathbb{I}(z) :=
	    \begin{cases}
	      z &: z\geq 0,\\
	      +\infty &: z<0.
	    \end{cases}
\]
\end{thmalpha}

As a last result, we establish the following large deviations principle for the $\ell_\infty$-norm.

\begin{thmalpha}\label{thm:maxldp}
	The sequence $(\frac{n}{\log n}\Vert Z_n\Vert_{\infty})_{n\in\mathbb{N}} $ satisfies an LDP with speed $ s_n=\log n $ and rate function
	\[\mathbb{I}(z) :=
	    \begin{cases}
	      z-1 &: z\geq 1,\\
	      +\infty &: z<1.
	    \end{cases}
	\]
\end{thmalpha}

Let us briefly compare the limit theorems for the case $q=+\infty$. The statement of the non-central limit theorem (Theorem \ref{thm:non-central limit theorem}) implies for every $x>0$ the following behaviour:
\[
\lim_{n\to\infty}\log\, \Pro\left[\log n\Bigl(\frac{n}{\log n}\|Z_n\|_{\infty}-\frac{\log n-1}{\log n}\Bigr)>x\right]=\log(1-\exp(-e^{-x}))
\]
and
\[
\lim_{n\to\infty}\log \,\Pro\left[\log n\Bigl(\frac{n}{\log n}\|Z_n\|_{\infty}-\frac{\log n-1}{\log n}\Bigr)<-x\right]=-e^{x}.
\]
The LDP implies for sets of the form $(x,\infty)$ with $x>0$ that
\[
\lim_{n\to\infty} \frac{1}{\log n}\log\, \Pro\left[\frac{n}{\log n}\|Z_n\|_{\infty}-1>x\right]=-x,
\]
and for sets $(-\infty,-x)$ that
\[
\lim_{n\to\infty} \frac{1}{\log n}\log\, \Pro\left[\frac{n}{\log n}\|Z_n\|_{\infty}-1<-x\right]=-\infty.
\]
The moderate deviations estimates of Theorem \ref{thm:moderate} provide information on the asymptotic likelihood of events for intermediate speeds $s_n$ slower than $\log n$ with corresponding slower convergence speed. In contrast to the MDP for sums of independent and identically distributed random variables, see for example Lemma \ref{lem:mdpforsums} in the next section, the rate function does not seem to reflect the limiting Gumbel distribution.

\medskip

The rest of the paper is organized as follows. In Section \ref{sec:prelim}, we collect some background material on large deviations and introduce the notation we use throughout this paper. Section \ref{sec:poofs} is then devoted to the proofs of Theorems \ref{thm:berryesseen}, \ref{thm:non-central limit theorem}, \ref{thm:moderate}, and \ref{thm:maxldp}. In Section \ref{sec:comparison}, we compare the LDP for the simplex with the one for $\ell_p^n$-balls, in particular the one for the crosspolytope. In the final section, we present an alternative route to the central limit theorem (Corollary \ref{cor:cltqnorm}) using empirical process methods.

\section{Preliminaries}\label{sec:prelim}

\subsection{General notation}

Let $n\in\mathbb{N}$. Given $1\leq q\leq \infty$ and a vector $x=(x_1,\ldots,x_n)\in\mathbb{R}^n$, we write
\[
\norm{x}_q=
\begin{cases}
\Bigl(\sum_{i=1}^n \abs{x_i}^q\Bigr)^{1/q} &: q<\infty,\\
\max_{1\leq i\leq n} \abs{x_i} &: q=\infty.
\end{cases}
\]
We will assume that all random quantities are defined on a common probability space $(\Omega,\Sigma,\mathbb{P})$ and we write $\pr{\,\cdot\,}$ and $\ee{\,\cdot\,}$ for the probability of an event and the expectation of an (integrable) random variable, respectively. For a sequence of independent and identically distributed (i.i.d.) random vectors $(X_i)_{i\in\mathbb{N}}$ we denote by $\bar{X}_n=\frac{1}{n}\sum_{i=1}^n X_i $ the empirical average. Throughout, $E,E_1,E_2,\ldots$ will be independent (standard) exponential random variables having rate one and $\bar{E}_n=\frac{1}{n}\sum_{i=1}^n E_i$. Note that $\pr{\bar{E}_n=0}=0$ and thus we can ignore this event in our analysis. By $\mathcal{N}(\mu,\Sigma)$ we denote the (multivariate) Gaussian distribution with mean $\mu$ and covariance matrix $\Sigma$. If a random variable $N$ is distributed according to $\mathcal{N}(\mu,\Sigma)$, we write $N\sim\mathcal{N}(\mu,\Sigma)$. With $\xrightarrow{d}$ and $\xrightarrow{\mathbb{P}}$ we indicate convergence in distribution and in probability, respectively. We say that a sequence of real-valued random variables $(X_n)_{n\in\mathbb{N}}$ satisfies a central limit theorem (CLT) if there exists a sequence $(a_n)_{n\in\mathbb{N}}$ of real numbers such that $\sqrt{n}(a_nX_n-1)\xrightarrow{d} N\sim \mathcal{N}(0,1)$ as $n\to\infty$. For further background material on asymptotic probability theory consult, for example, DasGupta \cite{D2008}. 

\subsection{Large and Moderate Deviations}\label{sec:LDPMDPprel}

In the following, we recall facts from the theory of large deviations as developed, for example, in \cite{DZ2010}.

A sequence $(X_n)_{n\in\mathbb{N}} $ of real-valued random variables is said to satisfy a large deviations principle (LDP) with speed $ (s_n)_{n\in\mathbb{N}}\subset (0,\infty) $ and rate function $ \mathbb{I}: \mathbb{R}\rightarrow [0,\infty] $ if $\mathbb{I}$ is lower semi-continuous, has compact level sets $\{z\in\mathbb{R}:\mathbb{I}(z)\leq \alpha\}, \alpha\in\mathbb{R},$ and if for all Borel sets $ A\subset \mathbb{R} $,
\[
 -\inf_{z\in A^{\circ}} \mathbb{I}(z)\leq \liminf_{n\to\infty} s_n^{-1}\log\,\pr{X_n\in A}\leq  \limsup_{n\to\infty} s_n^{-1}\log\,\pr{X_n\in A}\leq -\inf_{z\in\bar{A}}\mathbb{I}(z).
\]
Here $A^\circ$ denotes the interior and $\bar{A}$ the closure of $A$. For the empirical average of independent and identically distributed (real-valued) random variables an LDP holds according to Cram\'er's theorem, which we state next.

\begin{lemma}[\textbf{Cram\'ers theorem {\cite[Theorem 2.2.3]{DZ2010}}}]\label{lem:cramer}
Let $ X,X_1,X_2,\ldots $ be i.i.d. real-valued random variables. Assume that the origin is an interior point of the domain of the cumulant generating function $ \Lambda(u)=\log \,\mathbb{E}[\exp(uX)] $. Then the sequence of partial sums $ \bar{X}_n=\frac{1}{n} \sum_{i=1}^n X_i, n\in\mathbb{N},$ satisfies an LDP on $ \mathbb{R} $ with speed $ n $ and rate function $ \Lambda^*$, where $\Lambda^*(z)=\sup_{u\in\mathbb{R}}\bigl(uz-\Lambda(u)\bigr)$ for all $z\in\mathbb{R}$.
\end{lemma}

It is often useful to transfer an LDP for a sequence of random variables to another such sequence when they do not differ too much from each other. The following lemma provides a condition (called exponential equivalence) under which such an attempt is possible.

\begin{lemma}[\textbf{Exponential equivalence {\cite[Theorem 4.2.13]{DZ2010}}}]\label{lem:expequivalence}
	Let $ (X_n)_{n\in\mathbb{N}} $ and $(Y_n)_{n\in\mathbb{N}} $ be two sequences of real-valued random variables and assume that $ (X_n)_{n\in\mathbb{N}} $ satisfies an LDP with speed $ s_n $ and rate function $ \mathbb{I} $. If $ (X_n)_{n\in\mathbb{N}} $ and $ (Y_n)_{n\in\mathbb{N}} $ are exponentially equivalent at speed $s_n$, i.e., we have for any $ \delta>0 $ that
	\[
	 \limsup_{n\to\infty} s_n^{-1}\log\, \pr{|X_n-Y_n|>\delta}=-\infty,
	\]
	then $ (Y_n)_{n\in\mathbb{N}} $ satisfies an LDP with the same speed and rate function as $ (X_n)_{n\in\mathbb{N}} $.
\end{lemma}

A moderate deviations principle (MDP) is formally of the same nature as an LDP but operates on the scale between a limit theorem, such as a CLT, and a statement about convergence in probability, such as a law of large numbers. We shall need the following result for moderate deviations of the empirical average of i.i.d. random variables.

\begin{lemma}[\textbf{Moderate deviations {\cite[Theorem 3.7.1]{DZ2010}}}]\label{lem:mdpforsums}
Let $ X,X_1,X_2,\ldots $ be i.i.d. real-valued random variables with $\ee{X}=\mu$ and $\Var X=\sigma^2>0$. Assume that the origin is an interior point of the domain of the cumulant generating function $ \Lambda(u)=\log \,\mathbb{E}\exp(uX)$. Fix a sequence $(a_n)_{n\in\mathbb{N}}$ with $a_n\to 0$ and $na_n\to \infty$. Then, the sequence $\sqrt{na_n}(\bar{X}_n-\mu)$ satisfies an LDP on $ \mathbb{R} $ with speed $ a_n^{-1} $ and rate function $\mathbb{I}(z)=\frac{z^2}{2\sigma^2}.$
\end{lemma}

We provide some explanation. Under the assumptions of the previous lemma, by the (usual) central limit theorem, the CLT $\sqrt{n}(\bar{X}_n-\mu)\xrightarrow{d} N\sim  \mathcal{N}(0,\sigma^2)$ holds, as $n\to\infty$. In Lemma \ref{lem:mdpforsums}, the exponent of the density of $N$ is reflected in the rate function, and as the prefactor $\sqrt{na_n}$ becomes closer to $\sqrt{n}$, the speed $a_n^{-1}$ decreases. It must be stressed, however, that in general, as discussed in the Introduction, the rate function in an MDP may or may not reflect the limiting distribution.

\section{The proofs}\label{sec:poofs}

We shall now present the proofs of Theorems \ref{thm:berryesseen}, \ref{thm:non-central limit theorem}, \ref{thm:moderate}, and \ref{thm:maxldp}, and start with the Berry-Esseen-type central limit theorem followed by the non-central limit theorem together with the moderate and large deviations principles when $q=\infty$.

\subsection{Proof of the Berry-Esseen-CLT}

The general philosophy of the proof is similar to the one of Johnston and Prochno \cite{JP2019}. However, as already explained above, we shall use a connection to the asymptotic theory of sums of spacings.

The following lemma is a version of \cite[Lemma 4.1]{APT2019} (see also \cite[Lemma 2.8]{JP2019}).

\begin{lemma}\label{lem:gaussiantriangle}
For any real-valued random variables $X,Y$ and any $\varepsilon>0$ it holds that
\[
\sup_{x\in\R}\big| \mathbb{P}[Y\leq x]-\mathbb{P}[N\leq x]\big| \leq  \sup_{x\in\R}\vert \mathbb{P}[X\leq x]-\mathbb{P}[N\leq x]\vert +\mathbb{P}[\vert X-Y\vert>\varepsilon]+\frac{\varepsilon}{\sqrt{2\pi}},
\]
where $N$ is a standard Gaussian random variable. 
\end{lemma}
\begin{proof}
Let $x\in\R$ and $\varepsilon>0$. Then,
\begin{align*}
\mathbb{P}[Y\leq x]-\mathbb{P}[N\leq x]&\leq \mathbb{P}[Y\leq x, \vert X-Y\vert\leq \varepsilon] + \mathbb{P}[\vert X-Y\vert>\varepsilon]-\mathbb{P}[N\leq x]\\
&\leq \mathbb{P}[X\leq x+\varepsilon]+ \mathbb{P}[\vert X-Y\vert>\varepsilon] -\mathbb{P}[N\leq x]\\
&=\mathbb{P}[X\leq x+\varepsilon]-\mathbb{P}[N\leq x+\varepsilon]+\mathbb{P}[\vert X-Y\vert>\varepsilon]\\
&+\mathbb{P}[N\leq x+\varepsilon]-\mathbb{P}[N\leq x].
\end{align*}
Using that $\mathbb{P}[N\leq x+\varepsilon]-\mathbb{P}[N\leq x]\leq \varepsilon/\sqrt{2\pi}$ for all $x\in\R$, taking absolute values, and forming the supremum completes the proof.
\end{proof}

We shall need another lemma before we can derive the proof of Theorem \ref{thm:berryesseen} from a Berry-Esseen bound for sums of spacings to be stated subsequently. The lemma shows that similar to CLTs also Berry-Esseen-type bounds can be transfered by ``nice'' functions.

\begin{lemma}\label{lem:berryesseendeltamethod}
Let $(X_n)_{n\in\N}$ be a sequence of real-valued random variables. Suppose that there exist constants $\mu\in\mathbb{R}$ and $\sigma>0$ such that the Berry-Esseen-type bound
\[
\sup_{x\in\R}\Bigg|\mathbb{P}\left[\sqrt{n}\left(\frac{X_n-\mu}{\sigma}\right)\leq x\right]-\mathbb{P}[N\leq x]\Bigg|\leq a_n
\]
holds for some sequence $(a_n)_{n\in\mathbb{N}}$ and all $n\in\mathbb{N}$, where $N$ is a standard Gaussian random variable. If $g:\R\to\R$ is twice continuously differentiable at $\mu$ with $g'(\mu)>0$, then for some constant $C>0$ and all $n\in\mathbb{N}$,
\[
\sup_{x\in\R}\Bigg|\mathbb{P}\left[\sqrt{n}\left(\frac{g(X_n)-g(\mu)}{g'(\mu)\sigma}\right)\leq x\right]-\mathbb{P}[N\leq x]\Bigg|\leq C\max\left\{a_n,\frac{\log n}{\sqrt{n}}\right\}.
\]
\end{lemma}
\begin{proof}
Set $\xi_n:=\sqrt{n}(X_n-\mu)/\sigma$ and $\zeta_n:=\sqrt{n}(g(X_n)-g(\mu))/g'(\mu)\sigma$. We use Lemma \ref{lem:gaussiantriangle} to infer, for each $n\in\mathbb{N}$ and every $\varepsilon>0$,
\begin{equation}\label{eq:gaussiantriangleapplied}
\sup_{x\in\R}\vert \mathbb{P}[\zeta_n\leq x]-\mathbb{P}[N\leq x]\vert \leq  \sup_{x\in\R}\vert \mathbb{P}[\xi_n \leq x]-\mathbb{P}[N\leq x]\vert +\mathbb{P}[\vert \xi_n-\zeta_n\vert>\varepsilon]+\frac{\varepsilon}{\sqrt{2\pi}}.
\end{equation}
Fix $n\in\mathbb{N}$. We will estimate $\mathbb{P}[\vert \xi_n-\zeta_n\vert>\varepsilon]$ and choose $\varepsilon$ suitably.

Making use of the Taylor expansion of $g$ at $\mu$ yields that there exists some $\delta>0$ such that, for all $x\in (\mu-\delta,\mu+\delta)$, we have
\[
g(x)-g(\mu)=g'(\mu)(x-\mu)+\Phi(x-\mu),
\]
where $\Phi:\R\to\R$ is a function such that, for all $x\in (\mu-\delta,\mu+\delta)$, we have $\vert \Phi(x)\vert \leq M_g\vert x-\mu\vert^2$ for $M_g=\sup_{x\in (\mu-\delta,\mu+\delta)}\vert g''(x)\vert/2$. Thus, if $\vert X_n-\mu\vert<\delta$,
\[
\abs{g(X_n)-g(\mu)-g'(\mu)(X_n-\mu)}\leq M_g \abs{X_n-\mu}^2.
\]
We get, after division by $g'(\mu)\sigma$ and multiplication by $\sqrt{n}$,
\[
\vert \xi_n-\zeta_n\vert\leq \frac{M_g}{g'(\mu)\sigma}\sqrt{n}\vert X_n-\mu\vert^2.
\]
Therefore, for every $n\in\mathbb{N}$,
\begin{equation}\label{eq:tailbound}
\mathbb{P}[\vert \xi_n-\zeta_n\vert> \varepsilon]\leq \mathbb{P}\left[\vert\sqrt{n} (X_n-\mu)\vert>\sqrt{\frac{\varepsilon g'(\mu)\sigma\sqrt{n}}{M_g}}\right] + \mathbb{P}[\vert\sqrt{n}(X_n-\mu)\vert>\delta\sqrt{n}].
\end{equation}
If $M_g=0$, the first summand disappears and we can set $\varepsilon=a_n$. By the assumed bound and the symmetry of a Gaussian random variable, for each $n\in\mathbb{N}$ and every $x\in\mathbb{R}$,
\begin{align*}
\pr{\abs{\sqrt{n}(X_n-\mu)}>\sigma x}&=\pr{\sqrt{n}(X_n-\mu)<-\sigma x}+\pr{\sqrt{n}(X_n-\mu)>\sigma x}\\
&\leq 2\bigl(\pr{N\geq x}+a_n\bigr).
\end{align*}
Together with the bound $\pr{N\geq x}\leq e^{-x^2}$, the second summand in \eqref{eq:tailbound} is
\[
\pr{\abs{\sqrt{n}(X_n-\mu)}>\delta\sqrt{n}}\leq 2(e^{-\delta^2/\sigma^2 n}+a_n)\leq c\max\left\{a_n,\frac{\log n}{\sqrt{n}}\right\}
\]
for some $c>0$ independent of $n\in\mathbb{N}$. If $M_g>0$, by setting $\varepsilon=\sigma M_g g'(\mu)^{-1}n^{-1/2}\log n$ for each $n\in\mathbb{N}$, the first summand is
\[
\pr{\abs{\sqrt{n}(X_n-\mu)}>\sigma\sqrt{\log n}}\leq 2(n^{-1}+a_n)\leq c'\max\left\{a_n,\frac{\log n}{\sqrt{n}}\right\}
\]
for some constant $c'>0$ independent of $n$. This choice of $\varepsilon$ yields
\[
\mathbb{P}[\vert \xi_n-\zeta_n\vert> \sigma M_g g'(\mu)^{-1}n^{-1/2}\log n]\leq (c+c')\max\left\{a_n,\frac{\log n}{\sqrt{n}}\right\}.
\]
Together with inequality \eqref{eq:gaussiantriangleapplied} we have for all $n\in\mathbb{N}$
\[
\sup_{x\in\R}\vert \mathbb{P}[\zeta_n\leq x]-\mathbb{P}[N\leq x]\vert \leq  a_n +(c+c')\max\left\{a_n,\frac{\log n}{\sqrt{n}}\right\}+\frac{\sigma M_g}{\sqrt{2\pi} g'(\mu)}\frac{\log n}{\sqrt{n}},
\]
whereupon choosing $C>0$ suitably completes the proof.
\end{proof}

\begin{rmk}
There exist results in the literature which are similar to Lemma \ref{lem:berryesseendeltamethod}. For example, Theorem 11.6 in \cite{D2008} deals with the case of $X_n$ being empirical averages and $g$ a function with H\"older-continuous derivative. In this case one has $a_n=n^{-1/2}$ and the guaranteed bound for the modified sequence is of order $\frac{\log n}{\sqrt{n}}$. We do not know if this rate in Lemma \ref{lem:berryesseendeltamethod} can be improved in general.
\end{rmk}

Recall the definition of the spacings as defined in the introduction by
\[
G_{n,i}=U_{(i)}-U_{(i-1)}, \qquad i=1,\ldots,n,
\]
where $U_{(i)}$ is the $i^{\text{th}}$ order statistic of $U_1,\ldots,U_{n-1}$, sampled independently and uniformly from the unit interval, with the convention that $U_{(0)}=0, U_{(n)}=1$. Also recall that $E$ denotes an exponential random variable with rate $1$.

We deduce the following theorem from Mirakhmedov \cite{M2005} who refined a Berry-Esseen theorem due to Does and Klaassen \cite{DK1984}.

\begin{thm}\label{thm:berryesseenspacings}
Let $G_{n,1},\ldots, G_{n,n}$ be as above. Suppose $f:\R\to\R$ is measurable with $\ee{f(E)^3}<\infty$ and $\sigma^2:=\Var f(E)-\Cov(E,f(E))^2>0$. Then there exists a constant $C>0$ such that, for all $n\in\mathbb{N}$,
\[
\sup_{x\in\R}\Bigg|\mathbb{P}\left[\sqrt{n}\left(\frac{\frac{1}{n}\sum_{i=1}^nf(nG_{n,i})-\mathbb{E}[f(E)]}{\sigma}\right)\leq x\right]-\mathbb{P}[N\leq x]\Bigg| \leq \frac{C}{\sqrt{n}},
\]
where $N$ is a standard Gaussian random variable.
\end{thm}
\begin{proof}
For the convenience of the reader we mention the necessary modifications in order to derive Theorem \ref{thm:berryesseenspacings}. In \cite{M2005} set $k=1$ and $f_m=f$, for $m=1,\ldots,N$, giving $N=N'=n$ and $R_N(G)=\sum_{i=1}^nf(nG_{n,i})$. Then $Z_{m,k}=Y_m$, $m=1,\ldots,N$, is a standard exponential random variable. In the second line in \cite[Section 2]{M2005} the author appears to incorrectly redefine $Z_{N,k}=0$ if $k$ is an integer. The relevant quantities compute to $\rho=\Cov(f(E),E)(\Var f(E))^{-1/2},$ $\Var R_N(Z)=n\Var f(E)$ and $\sigma_n^2=(1-\rho^2)\Var R_N(Z)=n(\Var f(E)-\Cov(f(E),E)^2)$. The additional term $(u-1)\Cov(f(E),E)(\Var f(E))^{1/2}$ in $g_m(u)$ vanishes in the definition of $T_N(G)$ since $\sum_{i=1}^n G_{n,i}=1$. Under the assumptions stated, the application of \cite[Corollary 3]{M2005} is valid and completes the proof.
\end{proof}

The following results are preparations for the proof of Theorem \ref{thm:berryesseen}, carrying out more technical computations. Again, $E$ stands for a standard exponential random variable.

\begin{lemma}\label{lem:cov}
Let $1\leq q<\infty$. Then,
\[
\Cov(E,\abs{E-1}^q)=(q+1)\ee{\abs{E-1}^q}-1.
\]
\end{lemma}

For the proof of this result we need the following lemma.

\begin{lemma}\label{lem:derivative}
Let $1\leq q<\infty$. The function $x\mapsto \ee{\abs{E-x}^q}, x\in\R$, is continuously differentiable at $x=\ee{E}=1$ with derivative
\[
\ee{\abs{E-x}^q}'(1)= 1-\ee{\abs{E-1}^q}\quad \text{and}\quad \ee{\abs{E-1}^q}=e^{-1}\left(\Gamma(q+1)+\int_{0}^1 x^q e^x\text{d}x\right)
\]
with $\Gamma(x)=\int_0^{\infty}t^{x-1}e^{-t}\text{d}t$ being the Gamma function.
\end{lemma}
\begin{proof}
	We first compute
	\[
	\ee{\abs{E-1}^q}=\int_{0}^1 (1-x)^qe^{-x}\text{d}x+\int_1^{\infty} (x-1)^qe^{-x}\text{d}x.
	\]
	Using the substitutions $u=1-x$ and $u=x-1$ respectively, gives
	\[
	\int_{0}^1 (1-x)^qe^{-x}\text{d}x=e^{-1}\int_{0}^1 u^q e^u\text{d}u \quad \text{and} \quad \int_1^{\infty} (x-1)^qe^{-x}\text{d}x=e^{-1}\Gamma(q+1)
	\]
	and thus proves the second equality.
	
	It is a consequence of the dominated convergence theorem that
\begin{equation}\label{eq:diffmoment1}
\ee{\abs{E-x}^q}'(1)=q\Big(2\ee{(1-E)^{q-1}\textbf{1}_{(-\infty,1]}(E)} - \ee{\abs{ E-1}^{q-1}}\Big).
\end{equation}
	By substituting as before, the first term is
	\[
	2\ee{\abs{E-1}^{q-1}\textbf{1}_{(-\infty,1]}(E)}=2e^{-1}\int_0^1 x^{q-1}e^x\text{d}x.
	\]
	Therefore, by inserting for $ \ee{\abs{E-1}^{q-1}} $, we have
	\begin{equation}\label{eq:diffmoment2}
	\ee{\abs{E-x}^q}'(1)=\frac{q}{e}\left(\int_0^1 x^{q-1}e^x\text{d}x-\Gamma(q)\right).
	\end{equation}
	Integration by parts yields
	\[
	\int_0^1 x^q e^x\text{d}x=e-q\int_0^1 x^{q-1} e^x\text{d}x,
	\]
	and thus, by means of $\Gamma(q+1)=q\Gamma(q)$, we have
	\[
	\mathbb{E}[\vert E-1\vert^q]=e^{-1}\left(\Gamma(q+1)+\int_0^1 x^q e^x\text{d}x\right)=\frac{q}{e}\left(\Gamma(q)-\int_0^1 x^{q-1} e^x\text{d}x\right)+1.
	\]
	Comparing this with \eqref{eq:diffmoment2} proves the first equality.
\end{proof}

We are now ready to compute the covariance.
\begin{proof}[Proof of Lemma \ref{lem:cov}]
By straightforward calculation,
\begin{align*}
\Cov(E,\abs{E-1}^q)&
=\ee{E\abs{E-1}^q}-\ee{E}\ee{\abs{E-1}^q}
=\ee{(E-1)\abs{E-1}^q}\\
&=-\ee{(1-E)^{q+1}\textbf{1}_{(-\infty,1)}(E)}+\ee{(E-1)^{q+1}\textbf{1}_{[1,\infty)}(E)}.
\end{align*}
Taking into account Lemma \ref{lem:derivative} and \eqref{eq:diffmoment1} it follows that
\[
\Cov(E,\abs{E-1}^q)=-(q+2)^{-1}\ee{\abs{E-x}^{q+2}}'(1)=-(q+2)^{-1}\bigl(1-\ee{\abs{E-1}^{q+2}}\bigr).
\]
By repeated partial integration it holds that
\[
1-\ee{\abs{E-1}^{q+2}}=-(q+2)(q+1)\ee{\abs{E-1}^q}+(q+2).
\]
Consequently, the result follows.
\end{proof}

\begin{proof}[Proof of Theorem \ref{thm:berryesseen}]
Using the connection to the spacings $G_{n,1},\ldots,G_{n,n}$ we have
\[
\|Z_n\|_q^q=\sum_{i=1}^n \Bigl|Z_{n,i}-\frac{1}{n}\Bigr|^q\overset{d}{=}\sum_{i=1}^n \Bigl|G_{n,i}-\frac{1}{n}\Bigr|^q=n^{-q}\sum_{i=1}^n |nG_{n,i}-1|^q.
\]
That is, with
\[
n^{q-1}\|Z_n\|_q^q\overset{d}{=}\frac{1}{n}\sum_{i=1}^n f(nG_{n,i}),
\]
we are in the situation of Theorem \ref{thm:berryesseenspacings} with $f(x)=|x-1|^q$, which is not of the form $x\mapsto ax+b$ for any $a,b\in\mathbb{R}$ and all $x\in\mathbb{R}$. Because of this and the fact that the exponential distribution has finite moments of all orders, the assumptions of Theorem \ref{thm:berryesseenspacings} are satisfied and there exists a constant $c_q>0$ only depending on $q$ such that, for all $n\in\mathbb{N}$,
\[
\sup_{x\in\R}\Bigg|\mathbb{P}\left[\sqrt{n}\left(\frac{n^{q-1}\Vert Z_n\Vert_q^q -\mu_q}{\sigma_q}\right)\leq x\right]-\mathbb{P}[N\leq x]\Bigg| \leq \frac{c_q}{\sqrt{n}},
\]
with
\[
\sigma_q^2=\Var \vert E-1\vert^q -\Cov(E,|E-1|^q)^2=\mu_{2q}-\mu_q^2-((q+1)\mu_q-1)^2
\]
by Lemma \ref{lem:cov}. Applying Lemma \ref{lem:berryesseendeltamethod} with $g(x)=x^{1/q}$ and $g'(\mu_q)=q^{-1}\mu_q^{1/q-1}$ and rearranging terms concludes the proof.
\end{proof}

\begin{rmk}
Using the so-called delta method (see, e.g., Lemma \ref{lem:deltamethod}), one can alternatively derive the central limit theorem stated as Corollary \ref{cor:cltqnorm} also from a CLT for sum-functions of spacings proved by Holst \cite{H1979}.
\end{rmk}

\begin{rmk}
We briefly want to put $\mu_q$ in a more accessible form and compare it to the centering constant in the central limit theorem \cite[Theorem 1.1]{KPT2019_I} stating for $1<q\leq \infty$ and random vectors $Y_n$, which are uniformly distributed in the $\ell_1^n$-ball, that
\[
\sqrt{n}\left(n^{1-1/q}\norm{Y_n}_q M_1(q)^{-1/q}-1\right)\limd N\bigl(0,C_1(q,q)\bigr),
\]
where
\[
M_1(q):=\Gamma(q+1)\qquad  \text{and}\qquad C_1(q,q)=q^{-2}\left(\frac{\Gamma(2q+1)}{\Gamma(q+1)^2}-1\right)-1.
\]

As can be seen from Corollary \ref{cor:cltqnorm}, the same rate of $n^{1-1/q}$ appears. With repeated partial integration one can derive for integral $q$ that
\[
\mu_q=\ee{\abs{ E-1}^q}=
\begin{cases}
\ee{(E-1)^q} &: q \text{ even},\\
2e^{-1}q!-\ee{(E-1)^q} &: q \text{ odd},
\end{cases}
\]
where $\ee{(E-1)^q}$ equals the subfactorial $!q=q!\sum_{i=0}^q\frac{(-1)^i}{i!}$, which is also the nearest integer to $e^{-1}q!$. This is roughly by a factor of $e$ smaller than $M_1(q)=q!$.
\end{rmk}

\subsection{Proof of the non-central limit theorem}

	In the following, we give a proof of Theorem \ref{thm:non-central limit theorem} and analyze the limiting distribution of
	\[
	n\Vert Z_n\Vert_{\infty}=n\max_{1\leq i\leq n} |Z_{n,i}|=n\max_{1\leq i\leq n} \Big\vert \frac{E_i}{S_n}-\frac{1}{n}\Big\vert.
	\]
	Set $M_n:=\max_{1\leq i\leq n} E_i$ and let us recall (see, e.g., \cite[p. 125]{EKM1997}) that
	\[
	M_n-\log n=: G_n\limd G,
	\]
	where $G$ is standard Gumbel distributed. First, we prove that $\norm{Z_n}_{\infty}$ and
	\[
	T_n:=\max_{1\leq i\leq n} \left( \frac{E_i}{S_n}-\frac{1}{n}\right)=\frac{M_n}{S_n}-\frac{1}{n}
	\]
	are exponentially equivalent in the following sense.
	
\begin{lemma}\label{lem:equalmaxima}
	We have
	\[
	\lim_{n\to\infty}n^{-1}\log \,\mathbb{P}\bigl[\| Z_n\|_{\infty}\neq T_n\bigr]<0.
	\]
\end{lemma}
	\begin{proof}
		We first prove that everywhere except for $\{S_n=0\}$ we have the implication
		$$\Vert Z_n\Vert_{\infty}\neq T_n\Longrightarrow \frac{M_n}{S_n}<\frac{2}{n}.$$
		Note that $\Vert Z_n\Vert_{\infty}\geq T_n$ and if $\Vert Z_n\Vert_{\infty}\neq T_n$, there must be some index $i_0\in\{1,\ldots,n\}$ such that
		$$\Big\vert\frac{E_{i_0}}{S_n}-\frac{1}{n}\Big\vert>T_n=\frac{M_n}{S_n}-\frac{1}{n}.$$
		This can only occur if $\frac{E_{i_0}}{S_n}-\frac{1}{n}$ is negative (otherwise we would have equality), i.e.,
		$$\frac{1}{n}-\frac{E_{i_0}}{S_n}=\Big\vert\frac{E_{i_0}}{S_n}-\frac{1}{n}\Big\vert> \frac{M_n}{S_n}-\frac{1}{n}.$$
		Since $\frac{E_{i_0}}{S_n}\geq 0$, this gives the desired implication. Therefore,
		\[
		\mathbb{P}\bigl[\Vert Z_n\Vert_{\infty}\neq T_n\bigr]\leq \mathbb{P}\Big[M_n<2\bar{E}_n\Big]\leq \Pro[M_n<4]+\Pro[\bar{E}_n>2].
		\]
		By means of the inequality $1+x\leq e^x$, the first summand evaluates to
		\[
		\Pro[M_n<4]=(1-e^{-4})^n\leq e^{-e^{-4}n}
		\]
		 and by Cram\'{e}r's theorem, Lemma \ref{lem:cramer},
		\[
		\lim_{n\to\infty}n^{-1}\log\, \mathbb P[\bar{E}_n>2]<0. 	
		\]
		This completes the proof of the lemma.
	\end{proof}

The previous lemma implies that for any sequence $(a_n)_{n\in\N}$ the sequences $(a_n\|Z_n\|_{\infty})_{n\in\N}$ and $(a_nT_n)_{n\in\N}$ are exponentially equivalent at any speed $s_n$ with $s_n/n\to 0$ since then, for every $\delta>0$,
\[
\limsup_{n\to\infty}s_n^{-1}\log\,\prb{\abs{a_n\|Z_n\|_{\infty}-a_nT_n}>\delta}\leq \limsup_{n\to\infty} ns_n^{-1}n^{-1}\log\, \prb{\| Z_n\|_{\infty}\neq T_n}=-\infty.
\]
In particular, $n\Vert Z_n\Vert_{\infty}- nT_n\xrightarrow{\mathbb{P}}0$, and to establish Theorem \ref{thm:non-central limit theorem} it suffices to prove $nT_n-(\log n-1)\xrightarrow{d} G$.

By means of the connection to the spacings $G_n=(G_{n,1},\ldots,G_{n,n})$ we have
	\[
	nT_n\overset{d}{=}n\max_{1\leq i\leq n} G_{n,i}-1.
	\]
	We cite from Devroye \cite[Lemma 2.4]{D1981} the classical result
	\[
	n\max_{1\leq i\leq n} G_{n,i}-\log n\limd G,
	\]
	where $G$ is standard Gumbel distributed. Rewriting this for $nT_n$ gives
	\[
	nT_n-\log n+1\limd G.
	\]
	and completes the proof.\hfill $\Box$

\begin{rmk}
	We can read off from Theorem \ref{thm:non-central limit theorem} that the maximum norm of random points in the shifted simplex converges to zero at the rate $\frac{\log n}{n}$ with fluctuations of order $(\log n)^{-1}$. Uniformly distributed random vectors $Y_n$ in the $\ell_1^n$-ball exhibit a similar behaviour as can be read off from \cite[Theorem 1.1(c)]{KPT2019_I}. Namely,
	\[
	n\|Y_n\|_{\infty}-\log n\limd G.
	\]
\end{rmk}

\subsection{Proof of the MDP}
We will derive the proof of Theorem \ref{thm:moderate} from a result due to Devroye \cite[Lemma 3.2]{D1981}, which we rephrase in our notation.

\begin{lemma}\label{lem:devroye}
Let $(a_n)_{n\in\mathbb{N}}$ be a sequence of positive numbers satisfying $a_n\to 0$ and $a_n\log n\to\infty$, as $n\to\infty$. Then
\[
\lim_{n\to\infty} n^{a_n}\pr{\frac{n}{\log n}\frac{M_n}{S_n}-1>a_n}=1 \quad\text{ and }\quad\lim_{n\to\infty} \exp(n^{a_n})\pr{\frac{n}{\log n}\frac{M_n}{S_n}-1<-a_n}=1.
\]
\end{lemma}

In order to prove Theorem \ref{thm:moderate}, we restate this result in the following form.
\begin{lemma}\label{lem:mdpinterval}
Let $(s_n)_{n\in\N}$ be a positive sequence with $s_n\to \infty$ and $s_n/\log n\to 0$ as $n\to\infty$. Then, for any $x>0$,
\[
\lim_{n\to\infty}\frac{1}{s_n}\log \Pro\left[\frac{\log n}{s_n}\Bigl(\frac{n}{\log n}\|Z_n\|_{\infty}-1\Bigr)>x\right]=-x
\]
and
\[
\lim_{n\to\infty}\frac{1}{s_n}\log \Pro\left[\frac{\log n}{s_n}\Bigl(\frac{n}{\log n}\|Z_n\|_{\infty}-1\Bigr)<-x\right]=-\infty.
\]
\end{lemma}
\begin{proof}
Using the continuity of the logarithm and inserting $T_n=M_n/S_n-1/n$, we can deduce from Lemma \ref{lem:devroye} that both
\begin{equation}\label{eq:mdp1}
\lim_{n\to\infty} a_n\log n+\log\,\pr{\frac{n}{\log n}T_n-1>a_n-\frac{1}{\log n}}=0
\end{equation}
and
\begin{equation}\label{eq:mdp2}
\lim_{n\to\infty} n^{a_n}+\log\,\pr{\frac{n}{\log n}T_n-1<-a_n-\frac{1}{\log n}}=0
\end{equation}
for any such sequence $(a_n)_{n\in\mathbb{N}}$.

In the following, let $x>0$ be arbitrary. First, set $a_n=(s_n/\log n)x+1/\log n=(s_nx+1)/\log n$, where $(s_n)_{n\in\mathbb{N}}$ is any sequence with $s_n\to \infty$ and $s_n/\log n\rightarrow 0$. This choice of $(a_n)_{n\in\mathbb{N}}$ meets the assumptions of Lemma \ref{lem:devroye}. Then, inserting into Equation \eqref{eq:mdp1} gives
\[
\lim_{n\to\infty} s_nx+1+\log\,\pr{\frac{n}{\log n}T_n-1>\frac{s_n}{\log n}x}=0,
\]
which implies by considering the limit of the sequence divided by $s_n$ that
\[
x+\lim_{n\to\infty} \frac{1}{s_n}\log\,\pr{\frac{n}{\log n}T_n-1>\frac{s_n}{\log n}x}=0.
\]

For Equation \eqref{eq:mdp2} we choose $a_n=(s_n/\log n)x-1/\log n=(s_nx-1)/\log n$ for all $n\in\mathbb{N}$ such that $s_nx-1>0$ and set $a_n=1$ for all other $n\in\mathbb{N}$. Since $s_n\to\infty$, we have $s_nx-1>0$ for all $n\geq n_0$, where $n_0\in\mathbb{N}$ may depend on $x$, and only need to set finitely many terms $a_n=1$. This choice of $(a_n)_{n\in\mathbb{N}}$ satisfies the assumptions of Lemma \ref{lem:devroye}. Therefore,
\[
\lim_{n\to\infty} n^{(s_nx-1)/\log n}+\log\,\pr{\frac{n}{\log n}T_n-1<-\frac{s_n}{\log n}x}=0.
\]
Proceeding as before,
\[
\lim_{n\to\infty} \frac{1}{s_n}n^{(s_nx-1)/\log n}+\frac{1}{s_n}\log\,\pr{\frac{n}{\log n}T_n-1<-\frac{s_n}{\log n}x}=0.
\]
Noting that $n^{(s_nx-1)/\log n}/s_n=\exp(s_nx-1)/s_n\to\infty$, we have
\[
\lim_{n\to\infty} \frac{1}{s_n}\log\,\pr{\frac{n}{\log n}T_n-1<-\frac{s_n}{\log n}x}=-\infty.
\]
Since $s_n/n\to 0$, as $n\to\infty$, we can apply Lemma \ref{lem:equalmaxima} and rearrange terms to complete the proof of Lemma \ref{lem:mdpinterval}.
\end{proof}

We now deduce Theorem \ref{thm:moderate} using a standard technique in large deviations theory.

\begin{proof}[Proof of Theorem \ref{thm:moderate}]
Let $(s_n)_{n\in\mathbb{N}}$ be an arbitrary positive sequence with $s_n\to\infty$ and $s_n/\log n\to 0$. Theorem \ref{thm:moderate} follows if we can show for arbitrary open $U\subset\mathbb{R}$ and closed $C\subset \mathbb{R}$ the bounds
\begin{equation}\label{eq:mdpbounds}
\liminf_{n\to\infty}s_n^{-1}\log\,\pr{X_n\in U}\geq -\inf_{z\in U} \mathbb{I}(z)\quad \text{and} \quad\limsup_{n\to\infty}s_n^{-1}\log\,\pr{X_n\in C}\leq -\inf_{z\in C} \mathbb{I}(z),
\end{equation}
where we used the notation $X_n:=\frac{\log n}{s_n}\left(\frac{n}{\log n}\|Z_n\|_{\infty}-1\right)$. Recall that
\[
\mathbb{I}(z) :=
	    \begin{cases}
	      z &: z\geq 0,\\
	      +\infty &: z<0.
	    \end{cases}
\]
If $A\subset\mathbb{R}$, we use the notation
\[
A_-:=A\cap (-\infty,0] \quad \text{and}\quad A_+:=A\cap [0,\infty)
\]
as well as
\[
a_-:=\sup A_-\quad \text{and}\quad a_+=\inf A_+,
\]
such that $\inf_{z\in A}\mathbb{I}(z)=a_+$.

We first prove the upper bound in \eqref{eq:mdpbounds} and choose a closed set $C\subset\mathbb{R}$. If $C$ is empty, the upper bound is trivial as the probability of $X_n$ being in an empty set is zero. On the other hand, if $0\in C$, the infimum is zero and the upper bound is satisfied due to $\pr{X_n\in C}\leq 1$. Therefore, assume without loss of generality that at least one of $C_-$ and $C_+$ is not empty. If both are non-empty, then $c_-<0<c_+$ and
\[
C= C_-\cup C_+ \subset (-\infty,c_-]\cup [c_+,\infty)
\]
and $\pr{X_n\in C}\leq \pr{X_n<c_-}+\pr{X_n>c_+}$ for $n\in\mathbb{N}$. This also makes sense if $C_-$ is empty, i.e., $c_-=-\infty$ or if $C_+$ is empty, i.e., $c_+=\infty$, if we interpret $\pr{X_n<-\infty}=\pr{X_n>\infty}=0$. Because of the monotonicity of the logarithm and \cite[Lemma 1.2.15]{DZ2010} it holds that
\[
\limsup_{n\to\infty}s_n^{-1}\log\, \pr{X_n\in C}\leq \max\{M_-,M_+\}
\]
with
\[
M_-:= \limsup_{n\to\infty}s_n^{-1}\log\, \pr{X_n<c_-} \quad \text{and}\quad M_+:=\limsup_{n\to\infty}s_n^{-1}\log\, \pr{X_n>c_+}.
\]
Now the upper bound follows from Lemma \ref{lem:mdpinterval} since $M_-=-\infty$ and
\[
\max\{M_-,M_+\}=M_+=-c_+=-\inf_{z\in C}\mathbb{I}(z).
\]

We now prove the lower bound and choose an open set $U\subset\mathbb{R}$. If $U_+$ is empty, the infimum is $\infty$ and the lower bound is trivially satisfied. Assume therefore without loss of generality that $U_+$ is not empty. If $0\in U$, there exists $\delta>0$ such that $(-\delta,\delta)\subset U$ since $U$ is open. Hence,
\[
\pr{X_n\in U}\geq \pr{X_n\in (-\delta,\delta)}.
\]
In this case it remains to show that $\liminf_{n\to\infty}s_n^{-1}\log\,\pr{X_n\in (-\delta,\delta)}\geq 0$. This is implied by
\[
\lim_{n\to\infty}\pr{X_n\in(-\delta,\delta)}=1.
\]
Using the definition of $X_n$, let us write
\[
\lim_{n\to\infty}\pr{X_n\in(-\delta,\delta)}=\lim_{n\to\infty}\pr{(n\|Z_n\|_{\infty}-\log n)\in (-s_n\delta,s_n\delta)}.
\]
Since $s_n\to\infty$, for arbitary $N\in\mathbb{N}$ and for all large enough $n$ it holds that $s_n>N\delta^{-1}$. By the non-central limit theorem in Theorem \ref{thm:non-central limit theorem} this yields for a standard Gumbel distributed random variable $G$ the bound
\[
\lim_{n\to\infty}\pr{X_n\in(-\delta,\delta)}\geq \lim_{n\to\infty}\pr{(n\|Z_n\|_{\infty}-\log n+1)\in (1-N,1+N)}=\pr{G\in (1-N,1+N)}.
\]
Letting $N\to\infty$ gives $\lim_{n\to\infty}\pr{X_n\in (-\delta,\delta)}=1$ and this completes the case of $0\in U$.

If on the other hand $0\not\in U$, choose some $a\in U_+$ and, by the openness of $U$, a $\delta>0$ small enough such that $(a,a+\delta)\subset U_+$. Since $(a,a+\delta)=(a,\infty)\backslash (a+\delta,\infty)$, the superadditivity of the limit inferior gives
\begin{align*}
\liminf_{n\to\infty}s_n^{-1}\log\,\pr{X_n\in U}&\geq \liminf_{n\to\infty}s_n^{-1}\log\,\pr{X_n\in (a,a+\delta)}\\
&=\liminf_{n\to\infty}s_n^{-1}\log\,\big(\pr{X_n>a}-\pr{X_n>a+\delta}\big)\\
&\geq\liminf_{n\to\infty}s_n^{-1}\log\,\pr{X_n>a}+\liminf_{n\to\infty}s_n^{-1}\log\,\left(1-\frac{\pr{X_n>a+\delta}}{\pr{X_n>a}}\right).
\end{align*}
We show that the first summand is not less than $-\mathbb{I}(a)$ and that the second summand is in fact zero. Since $a\in U_+$ was arbitrary and $\inf_{z\in U}\mathbb{I}(z)=\inf_{z\in U_+}\mathbb{I}(z)$, the lower bound in \eqref{eq:mdpbounds} then follows.

By Lemma \ref{lem:mdpinterval}
\[
\lim_{n\to\infty}s_n^{-1}\log\,\pr{X_n>a}=-\mathbb{I}(a)=-a \quad \text{and}\quad \lim_{n\to\infty}s_n^{-1}\log\,\pr{X_n>a+\delta}=-\mathbb{I}(a+\delta)=-(a+\delta).
\]
Choose $\varepsilon>0$ smaller than $\delta/2$. For all large enough $n$ we have both
\[
\pr{X_n>a}\geq e^{-s_n(\mathbb{I}(a)+\varepsilon)}\quad \text{and}\quad \pr{X_n>a+\delta}\leq e^{-s_n(\mathbb{I}(a+\delta)-\varepsilon)}.
\]
Thus,
\[
\frac{\pr{X_n>a+\delta}}{\pr{X_n>a}}\leq e^{-s_n(\mathbb{I}(a+\delta)-\mathbb{I}(a)-2\varepsilon)}=e^{-s_n(\delta-2\varepsilon)}\to 0,
\]
and consequently,
\[
\liminf_{n\to\infty}s_n^{-1}\log\left(1-\frac{\pr{X_n>a+\delta}}{\pr{X_n>a}}\right)= 0.
\]
Therefore,
\[
\liminf_{n\to\infty}s_n^{-1}\log\,\pr{X_n\in U}\geq \liminf_{n\to\infty}s_n^{-1}\log\,\pr{X_n>a}\geq -\mathbb{I}(a)-\varepsilon.
\]
Letting $\varepsilon\to 0$ and noting that $a\in U_+$ was arbitrary yields
\[
\liminf_{n\to\infty}s_n^{-1}\log\,\pr{X_n\in U}\geq -\inf_{z\in U_+} \mathbb{I}(z)= -\inf_{z\in U}\mathbb{I}(z).
\]
This completes the proof of \eqref{eq:mdpbounds} and thus the proof of Theorem \ref{thm:moderate}.
\end{proof}

\subsection{Proof of the LDP}

For the proof of Theorem \ref{thm:maxldp} we use a rather general result on large deviations for maxima and minima due to Giuliano and Macci \cite[Proposition 3.1]{GM2014}. Recall that a function $H:\mathbb{R}\to\mathbb{R}$ is said to be regularly varying at $\infty $ of index $\alpha\in\mathbb{R}$ if, for all $u>0$, $\lim_{t\to\infty}\frac{H(ut)}{H(t)}=u^{\alpha}$.

\begin{lemma}\label{lem:ldpgiulianomacci}
	Let  $X,X_1,X_2,\ldots$ be i.i.d. real-valued random variables with $\pr{X\leq x}<1$ for all $x>0$ such that $-\log\,\pr{X>x} $ is regularly varying at $\infty$ of index $\alpha>0$ as a function of $x\in\mathbb{R}$. Choose $m_n\in\mathbb{R}$ such that $\pr{X>m_n}=\frac{1}{n}$ and set $M_n=\max_{1\leq i\leq n}X_i$. Then $ \big(\frac{M_n}{m_n}\big)_{n\geq 1} $ satisfies an LDP with speed $ s_n=\log n $ and good rate function
	\[\mathbb{I}(z) :=
	    \begin{cases}
	      z^{\alpha}-1 &: z\geq 1,\\
	      +\infty &: z<1.
	    \end{cases}
	\]
\end{lemma}

In order to prove Theorem \ref{thm:maxldp} we make use of this result, Lemma \ref{lem:expequivalence}, and the next lemma which states the exponential equivalence of the sequences in question.

\begin{lemma}\label{lem:expeqmax}
	Set $M_n=\max_{1\leq i\leq n} E_i$.  The sequences $\bigl(\frac{n}{\log n}\Vert Z_n\Vert_{\infty}\bigr)_{n\in\mathbb{N}} $ and $ \bigl(\frac{M_n}{\log n}\bigr)_{n\in\mathbb{N}} $ are exponentially equivalent at speed $ \log n $.
\end{lemma}
\begin{proof}
	This amounts to showing that, for every $ \delta>0 $,
	\[
	\limsup_{n\to\infty} \frac{1}{\log n}\log\, \prb{|n\|Z_n\|_{\infty}-M_n\bigr|>\delta\log n}=-\infty.
	\]
		Fix $ \delta>0 $. Then, for every $n\in\mathbb{N}$,
	\[
	\pr{\bigl|n\|Z_n\|_{\infty}-M_n\bigr|>\delta\log n}\leq \pr{\bigl|nT_n-M_n\bigr|>\delta\log n}+\prb{\|Z_{n}\|_{\infty}\neq T_n}.
	\]
	By Lemma \ref{lem:equalmaxima} the second summand satisfies
	\[
	\frac{1}{\log n}\log \,\prb{\Vert Z_{n}\Vert_{\infty}\neq T_n}\to -\infty \quad \text{as $n\to\infty$}.
	\]
	For the first summand we compute
	\[
	nT_n-M_n=M_n\bar{E}_n^{-1}-1-M_n=-(M_n(\bar{E}_n-1)\bar{E}_n^{-1}+1).
	\]
	Setting $A_n:=M_n(\bar{E}_n-1)\bar{E}_n^{-1}$, it follows that
	\[
	\pr{\bigl|nT_n-M_n\bigr|>\delta\log n}=\pr{\abs{A_n+1}>\delta\log n}.
	\]
	Whenever $n$ is large enough, $ \log n>2\delta^{-1} $ holds, and thus
	\[
	\prb{\abs{A_n+1}>\delta\log n}\leq \prb{\abs{A_n+1}>2}\leq \prb{\abs{A_n}>1}.
	\]
	Introducing $B_n:= n^{-1/4}M_n $ and $ C_n:=n^{1/4}(\bar{E}_n-1)\bar{E}_n^{-1} $, such that $A_n=B_nC_n$, gives
	\[
	\prb{\bigl|nT_n-M_n\bigr|>\delta\log n}\leq \prb{\abs{B_n}>1/2}+\prb{\abs{C_n}>2}.
	\]
	By the union bound the first summand satisfies
	\[
	\prb{\abs{B_n}>1/2}=\pr{M_n>n^{1/4}/2}\leq n\,\pr{E_1>n^{1/4}/2}=ne^{-n^{1/4}/2},
	\]	
	and thus
	\[
	\frac{1}{\log n}\log \,\prb{\abs{B_n}>1/2}\leq 1-\frac{n^{1/4}}{2\log n}\xrightarrow{n\to\infty} -\infty.
	\]
	The other summand can be estimated by means of
	\[
	\prb{\abs{C_n}>2}\leq \pr{\bar{E}_n<1/2}+\pr{n^{1/4}|\bar{E}_n-1|>1}.
	\]
	Cram\'{e}r's theorem (Lemma \ref{lem:cramer}) implies
	\[
	\frac{1}{\log n}\log \,\pr{\bar{E}_n<1/2}\xrightarrow{n\to\infty} -\infty.
	\]
	After splitting the second summand into
	\[
	\pr{n^{1/4}|\bar{E}_n-1|>1}=\pr{n^{1/4}(\bar{E}_n-1)>1}+\pr{n^{1/4}(\bar{E}_n-1)<-1},
	\]
	we apply the moderate deviations principle (Lemma \ref{lem:mdpforsums}) with $a_n=n^{-1/2}, n\in\mathbb{N},$ giving
	\[
	\lim_{n\to\infty}n^{-1/2}\log\,\pr{n^{1/4}(\bar{E}_n-1)>1}=-\frac{1}{2} \quad \text{ and } \quad \lim_{n\to\infty}n^{-1/2}\log\,\pr{n^{1/4}(\bar{E}_n-1)<-1}=-\frac{1}{2}.
	\]
	Consequently,
	\[
	\frac{1}{\log n}\log \,\pr{n^{1/4}|\bar{E}_n-1|>1}\xrightarrow{n\to\infty} -\infty,
	\]
	and thus
	\[
	\frac{1}{\log n}\log\,\prb{\abs{C_n}>2}\xrightarrow{n\to\infty} -\infty,
	\]
	completing the proof.
\end{proof}

\begin{rmk}
The choice of the sequence $n^{1/4}$ was arbitrary, any sequence growing faster than $\log n$ and slower than $\bigl(\frac{n}{\log n}\bigr)^{1/2}$ would have done the job.
\end{rmk}

\begin{proof}[Proof of Theorem \ref{thm:maxldp}]
We apply Lemma \ref{lem:ldpgiulianomacci} to the case of standard exponential random variables and verify the assumptions. We have $\pr{X\leq x}=1-e^{-x}<1$ for all $x\in\mathbb{R}$ and $-\log\pr{X>x}=x$, which is regularly varying at $\infty$ of index $\alpha=1$. We choose $m_n=\log n$ and thus obtain an LDP of speed $\log n$ and rate function as in Lemma \ref{lem:ldpgiulianomacci} with $\alpha=1$. By Lemma \ref{lem:expequivalence} the just proven Lemma \ref{lem:expeqmax} implies Theorem \ref{thm:maxldp}.
\end{proof}

\section{Comparison with the LDP for $\ell_p^n$-balls}\label{sec:comparison}

In \cite{KPT2019_I} LDPs for the $\ell_q$-norm of uniformly distributed points in $\ell_p^n$-balls were proven for the cases $1\leq p<\infty, 1\leq q<\infty$ and $p=\infty, 1\leq q\leq\infty$. In order to compare the LDP for the simplex with the one for the crosspolytope (i.e., the unit ball in $\ell_1^n$), we need to complete the picture presented in \cite{KPT2019_I} by deriving an LDP for the case $1\leq p<\infty$ and $q=\infty$. We first introduce the relevant concepts.

Let $1\leq p<\infty$ and $\mathbb{B}_p^n:=\{x\in \mathbb{R}^n: \|x\|_p\leq 1\}$ be the $\ell_p^n$-unit ball. Let $Z_n$ be uniformly distributed in $\mathbb{B}_p^n$ for all $n\in\N$. It is known through the work of Schechtman and Zinn \cite{SchechtmanZinn} that
\[
Z_n\overset{d}{=} U^{1/n}\frac{Y_n'}{\| Y_n'\|_p},
\]
where $U$ is uniformly distributed on $[0,1]$ and independent of $Y_n':=(Y_1,\ldots,Y_n)$ which has i.i.d. $p$-generalized Gaussian distributed entries. Here, we say that $Y_1$ is distributed according to the $p$-generalized Gaussian distribution if, for all $x\in\mathbb{R}$,
\[
\mathbb{P}[Y_1\leq x]=\int_{-\infty}^x f_{p}(y)\text{d}y,\quad \text{where}\quad f_p(y):=c_p e^{-|y|^p/p} \text{ for } y\in\mathbb{R},
\]
with
\[
c_p:=\frac{1}{2p^{1/p}\Gamma(1+1/p)}.
\]

We will show the following large deviations principle for $\|Z_n\|_{\infty}$.
\begin{thm}\label{thm:ldpmaxlp}
Let $1\leq p<\infty$. The sequence $\Bigl(\bigl(\frac{n}{p\log n}\bigr)^{1/p}\|Z_n\|_{\infty}\Bigr)_{n\in\mathbb{N}}$ satisfies an LDP with speed $\log n$ and rate function
\[
\mathbb{I}(z):=
\begin{cases}
z^p-1 &:z\geq 1,\\
\infty &:z<1.
\end{cases}
\]
\end{thm}

\begin{rmk}
Let us remark that the scaling $\bigl(\frac{n}{p\log n}\bigr)^{1/p}$ is identical to the non-central limit theorem \cite[Theorem 1.1.(c)]{KPT2019_I}. In view of the scaling for other $\ell_q$-norms, this additional logarithmic part is somehow natural when $q=\infty$. The rate function $\mathbb{I}$ which we obtain, is structurally similar to the one for the case $1\leq p<\infty, p<q<\infty$ as in \cite[Theorem 1.3]{KPT2019_I}.
\end{rmk}

\begin{rmk}
We also want to compare Theorem \ref{thm:maxldp} to the result obtained by Schechtman and Zinn \cite[p. 223] {SchechtmanZinn}, who proved in our notation that
\begin{equation}\label{eq:schechtmanzinn}
\pr{n^{1/p}\|Z_n\|_{\infty}>x}\leq e^{-\gamma x^p/p}\quad \text{ for all }x>\tau(\log n)^{1/p},
\end{equation}
where $\gamma,\tau>0$ are constant that do not depend on $n$. Theorem \ref{thm:ldpmaxlp} implies for $z\geq 1$, 
\[
\lim_{n\to\infty}\frac{1}{\log n}\log \,\pr{n^{1/p}\|Z_n\|_{\infty}>z(p\log n)^{1/p}} =-(z^p-1),
\]
which, in contrast to \eqref{eq:schechtmanzinn}, gives precise asymptotics for $\log \,\pr{n^{1/p}\|Z_n\|_{\infty}>x}$ and $x\geq p^{1/p}(\log n)^{1/p}$.
\end{rmk}

The proof of Theorem \ref{thm:ldpmaxlp} will follow from the following lemmata. In the same way as in \cite[Lemma 4.2]{GKR2017} one can prove tail asymptotics for the $p$-generalized Gaussian distribution.
\begin{lemma}\label{lem:pgaussiantail}
For all $1\leq p<\infty$ and $x>0$, we have
\[
\frac{x}{x^p+p}e^{-x^p/p}\leq \int_x^{\infty} e^{-y^p/p}\text{d}y\leq \frac{1}{x^{p-1}}e^{-x^p/p}.
\]
\end{lemma}

 Define, for each $n\in\mathbb{N}$, the number $m_n(p)>0$ by
\[
\mathbb{P}[|Y|> m_n(p)]=\frac{1}{n},
\]
where $Y$ is a $p$-generalized Gaussian random variable. From Lemma \ref{lem:ldpgiulianomacci} we can deduce the following result for $M_n(p):=\max_{1\leq i\leq n} |Y_i|.$

\begin{lemma}\label{lem:giulianomaccicor}
Let $1\leq p<\infty$. The sequence  $(M_n(p) m_n(p)^{-1})_{n\in\mathbb{N}}$ satisfies an LDP with speed $\log n$ and rate function
\[
\mathbb{I}(z):=
\begin{cases}
z^p-1 &:z\geq 1,\\
\infty &:z<1.
\end{cases}
\]
\end{lemma}
\begin{proof}
We check the assumptions of Lemma \ref{lem:ldpgiulianomacci}, and note that by the symmetry of the generalized Gaussian distribution, $\prb{|Y|>x}=2\pr{Y>x}>0$ for every $x>0$. In order to check if $-\log \pr{|Y|>x}$ is regularly varying at $\infty$, we compute for every $u>0$,
\[
\lim_{t\to\infty}\frac{\log \,\prb{|Y|>ut}}{\log\, \prb{|Y|>t}}=\lim_{t\to\infty}\frac{\log 2+\log\, \pr{Y>ut}}{\log 2+\log\, \pr{Y>t}}=\lim_{t\to\infty} \frac{(ut)^p}{t^p}=u^p,
\]
since, by Lemma \ref{lem:pgaussiantail}, $\log \,\prb{|Y|>x}$ is asymptotically equivalent to $-x^p/p$, as $x\to\infty$. Thus, the assumptions of Lemma \ref{lem:ldpgiulianomacci} are satisfied with $\alpha=\alpha_p:=p>0$ and we can complete the proof.
\end{proof}

In the following, we use for any two positive sequences $(a_n)_{n\in\mathbb{N}}$ and $(b_n)_{n\in\mathbb{N}}$ the notation $a_n\sim b_n$ if $\lim_{n\to\infty}\frac{a_n}{b_n}=1$. By means of the following lemma, we show exponential equivalence between $(M_n(p) m_n(p)^{-1})_{n\in\mathbb{N}}$ and $\Bigl(\bigl(\frac{n}{p\log n}\bigr)^{1/p}\|Z_n\|_{\infty}\Bigr)_{n\in\mathbb{N}}$ in two steps and complete the proof of Theorem \ref{thm:ldpmaxlp} with the help of Lemma \ref{lem:expequivalence}.

\begin{lemma}\label{lem:mnasymptotics}
Let $1\leq p<\infty$. With $m_n(p)$ as defined above, $m_n(p)\sim p^{1/p}(\log n)^{1/p}$, as $n\to\infty$.
\end{lemma}
\begin{proof}
This is a straightforward application of Lemma \ref{lem:pgaussiantail} as
\[
\frac{1}{n}=2\mathbb{P}[Y>m_n]=2c_p\int_{m_n(p)}^{\infty} e^{-y^p/p}\text{d}y\sim 2c_p m_n(p)^{-(p-1)}e^{-m_n(p)^p/p},
\]
when viewing both sides as sequences in $n\in\mathbb{N}$. Therefore, $2c_p n\sim m_n(p)^{p-1}e^{m_n^p/p},$ and taking logarithms gives
\[
\log(2c_p)+\log n\sim (p-1)\log m_n(p)+m_n(p)^p/p.
\]
Neglecting the asymptotically vanishing factors $\log(2c_p)$ and $\log m_n(p)$  yields $\log n\sim m_n(p)^p/p$, which completes the proof.
\end{proof}

\begin{lemma}\label{lem:expequivalenceldplp}
Let $1\leq p<\infty$. The sequences $(M_n(p) m_n(p)^{-1})_{n\in\mathbb{N}}$ and $(n^{1/p}\|Z_n\|_{\infty}m_n(p)^{-1})_{n\in\mathbb{N}}$ are exponentially equivalent at speed $\log n$.
\end{lemma}
\begin{proof}
Note that
\[
n^{1/p}\|Z_n\|_{\infty}m_n(p)^{-1}\overset{d}{=}U^{1/n}\frac{M_n(p) m_n(p)^{-1}}{\left(\frac{1}{n}\sum_{i=1}^n|Y_i|^p\right)^{1/p}},
\]
with the $Y_i's$ being independent and $p$-generalized Gaussians. Then the exponential equivalence is proved using Cram\'er's theorem as in the proof of \cite[Theorem 1.3]{KPT2019_I}. It remains to note that, for $0<\varepsilon\leq\delta$, we have by Lemma \ref{lem:giulianomaccicor}
\[
\limsup_{n\to\infty} \frac{1}{\log n}\log\, \mathbb{P}\left[M_n(p) m_n(p)^{-1}>\frac{\delta}{\varepsilon}\right]=-\left(\Bigl(\frac{\delta}{\varepsilon}\Bigr)^p-1\right),
\]
which tends to $-\infty$ as $\varepsilon\to 0$.
\end{proof}

We need one more lemma,  which we state in a general form.
\begin{lemma}\label{lem:expequivalenceseq}
Assume that a sequence of real-valued random variables $(X_n)_{n\in\mathbb{N}}$ satisfies an LDP with some speed $s_n$ and rate function $\mathbb{I}:\mathbb{R}\to [0,\infty]$ with $\lim_{x\to\pm\infty}\mathbb{I}(x)=\infty$. If $(a_n)_{n\in\mathbb{N}}$ is a sequence with $a_n\sim 1$, the sequence $(a_n X_n)_{n\in\N}$ satisfies an LDP with the same speed and rate function.
\end{lemma}
\begin{proof}
We establish exponential equivalence and note that for any $\delta,\varepsilon>0$,
\[
\prb{|X_n-a_n X_n|>\delta}\leq \pr{|X_n|>\frac{\delta}{\varepsilon}}+\prb{|1-a_n|>\varepsilon}.
\]
Since $a_n\to 1$ as $n\to\infty$, one can choose for any $\varepsilon>0$ an $n_0\in\mathbb{N}$ such that the constant random variable $a_n$ satisfies $|1-a_n|\leq \varepsilon$ if $n\geq n_0$. Therefore, for every $\varepsilon>0$,
\[
\limsup_{n\to\infty}s_n^{-1}\log\, \prb{|1-a_n|>\varepsilon}=-\infty.
\]
By the assumed LDP, we have
\[
\limsup_{n\to\infty}s_n^{-1}\log\,  \pr{|X_n|>\frac{\delta}{\varepsilon}}=-\inf_{|x|>\delta\varepsilon^{-1}}\mathbb{I}(x).
\]
If we let $\varepsilon\to 0$, this becomes $-\infty$ since $\lim_{x\to\pm\infty}\mathbb{I}(x)=\infty$. With this, the limit
\[
\limsup_{n\to\infty}s_n^{-1}\log\, \prb{|X_n-a_n X_n|>\delta}= -\infty
\]
is established and the proof is complete by Lemma \ref{lem:expequivalence}.
\end{proof}

\begin{proof}[Proof of Theorem \ref{thm:ldpmaxlp}]
By Lemma \ref{lem:giulianomaccicor} the sequence $(M_n(p)m_n(p)^{-1})_{n\in\mathbb{N}}$ satisfies the required LDP and so does $(n^{1/p}\|Z_n\|_{\infty}m_n(p)^{-1})_{n\in\mathbb{N}}$ by Lemma \ref{lem:expequivalenceldplp}. Since Lemma \ref{lem:mnasymptotics} gives $m_n(p)(p\log n)^{-1/p}\sim 1 $, we can apply Lemma \ref{lem:expequivalenceseq} to conclude the proof.
\end{proof}

\section{Another route to the CLT}

The route to obtain the Berry-Esseen-type central limit theorem, which we chose to present, used results from the asymptotic theory of sums of random spacings. While this is, of course, a nice trick to obtain the desired result for the regular simplex, it is usually not be applicable in other situations. However, there is another way to obtain the central limit theorem for random points in a regular simplex without rates of convergence, which we shall elaborate on now. Given the increased interest in central limit phenomena for random geometric systems in high dimensions, we consider this to be of independent interest.


More precisely, we want to provide some intuition into how the theory of empirical processes may be of help in proving a central limit theorem of the form
\begin{equation}\label{eq:generalclt}
\sqrt{n}\left(\frac{1}{n}\sum_{i=1}^n h(X_i,\bar{X}_n)-\ee{h(X_1,\mu)}\right)\limd N\sim \mathcal{N}(0,1),
\end{equation}
where $(X_i)_{i\in\mathbb{N}}$ is a suitable sequence of i.i.d.\ random variables with mean $\mu$, $\bar{X}_n=\frac{1}{n}\sum_{i=1}^n X_i$, and $h$ is a suitable function. Pollard \cite{P1989} and van der Vaart \cite[Example 19.25]{V1998} deal with this matter and applied this to the case $h(x,t)=\abs{x-t}$. We extend this to $h(x,t)=\abs{x-t}^q$ with $1\leq q<\infty$ and in the following sketch the proof for Corollary \ref{cor:cltqnorm} via empirical process methods.

We first reduce Corollary \ref{cor:cltqnorm} to a statement of the form \eqref{eq:generalclt}. For this, and also for proving the limit theorem, we shall use the multivariate delta method as presented in the next lemma, see, e.g., DasGupta \cite[Theorem 3.7]{D2008}.

\begin{lemma}\label{lem:deltamethod}
	Let $ (X_n)_{n\in\mathbb{N}}$ be a sequence of $ k $-dimensional random vectors such that $ \sqrt{n}(X_n-\mu)$ converges in distribution to a centered Gaussian random vector with covariance matrix $\Sigma$ and let $ g:\mathbb{R}^k\rightarrow \mathbb{R}^n $ be continuously differentiable at $ \mu $ with Jacobian matrix $ J_g $. Then
	\[
	\sqrt{n}\bigl(g(X_n)-g(\mu)\bigr)\limd N\sim \mathcal{N}(0, J_g\Sigma J_g^T ),
	\]
	provided $ J_g\Sigma J_g^T $ is positive definite.
\end{lemma}

With the delta method applied to the function $g(x)=x^{1/q}$, it is sufficient to prove a CLT for $n^{q-1}\|Z_n\|_q^q$. We have the identity
\[
n^{q-1}\|Z_n\|_q^q \overset{d}{=}\bar{E}_n^{-q}\frac{1}{n}\sum_{i=1}^n \abs{E_i-\bar{E}_n}^q.
\]
The additional factor $\bar{E}_n^{-q}$ compared to the sum in \eqref{eq:generalclt} with $h(x,t)=\abs{x-t}^q$ is a slight nuisance but does not complicate the analysis. We proceed by writing
\begin{align*}
\sqrt{n}\left(\bar{E}_n^{-q}\frac{1}{n}\sum_{i=1}^n \abs{E_i-\bar{E}_n}^q-\ee{\abs{E-1}^q}\right)=\sqrt{n}\left(Y_n-\ee{\abs{E-1}^q}\right)+R_n,
\end{align*}
where
\[
Y_n:=\bar{E}_n^{-q}\left(\frac{1}{n}\sum_{i=1}^n \abs{E_i-1}^q+\ee{\abs{E-\bar{E}_n}^q}-\ee{\abs{E-1}^q}\right)
\]
and
\[
R_n:=\bar{E}_n^{-q}\sqrt{n}\left(\frac{1}{n}\sum_{i=1}^n \Bigl(\abs{E_i-\bar{E}_n}^q-\abs{E_i-1}^q\Bigr)-\mathbb{E}\Bigl[\abs{E-\bar{E}_n}^q-\abs{E-1}^q\Bigr]\right).
\]
Using the delta method (Lemma \ref{lem:deltamethod}) we now derive a CLT for $Y_n$ and then describe how empirical process theory may be employed to show that the remainder $R_n$ converges to zero in probability. By Slutsky's theorem (see, e.g., \cite[Theorem 1.5]{D2008}) this proves the CLT for $n^{q-1}\|Z_n\|_q^q$.

Recall that $\mu_q=\ee{\abs{E-1}^q}$. For each $n\in\mathbb{N}$ we have $Y_n=F(\bar{E}_n,\frac{1}{n}\sum_{i=1}^n\abs{E_i-1}^q)$ for the function $F(x,y)=x^{-q}\left(y+\ee{\abs{E-x}^q}-\mu_q\right)$. Also, as partial sums of independent random vectors, the sequence $(\bar{E}_n,\frac{1}{n}\sum_{i=1}^n\abs{E_i-1}^q)^T$ satisfies a multivariate CLT with mean $(1,\mu_q)$ and covariance matrix
\[
\Sigma=\left(
\begin{array}{c c}
1 & \Cov(E,\abs{E-1}^q)\\
\Cov(E,\abs{E-1}^q) & \Var(\abs{E-1}^q)
\end{array}
\right).
\]
Because of Lemma \ref{lem:derivative}, the function $F$ has Jacobian $J=(1-(q+1)\mu_q,1)^T$ at $(1,\mu_q)$, and by the delta method and the formula for the covariance stated in Lemma \ref{lem:cov},
\[
\sqrt{n}\left(Y_n-\ee{\abs{E-1}^q}\right)\limd N\sim \mathcal{N}(0,\sigma_q^2)
\]
with
\begin{align*}
\sigma_q^2=J^T\Sigma J&=\Var(\abs{E-1}^q)+(1-(q+1)\mu_q)^2+2(1-(q+1)\mu_q)\Cov(E,\abs{E-1}^q)\\
&=\Var(\abs{E-1}^q)-\Cov(E,\abs{E-1}^q)^2.
\end{align*}

We will now analyze the remainder. Since $\bar{E}_n^{-q}\to 1$ in probability, as $n\to\infty$, it is sufficient, by Slutsky's theorem, to show that $R_n\bar{E}_n^{q}$ tends to zero in probability as well. To this end, we first recall notions which characterize the size and the `well-behavedness' of a class of functions with respect to uniform central limit theorems.

In the following, we will consider on a set of real-valued functions $\mathcal{F}$, each defined on some index set $S$, the pseudometric induced by the $L_2(Q)$-norm, where $Q$ is a finite measure on $S$. We write $D(\varepsilon,\mathcal{F},Q)$ for the corresponding packing number, i.e., $D(\varepsilon,\mathcal{F},Q)$ is the largest number $N$ such that there are functions $f_1,\ldots,f_N\in\mathcal{F}$ with $\norm{f_i-f_j}_{L_2(Q)}>\varepsilon$ for $i\neq j$.

Let $F:S\to\mathbb{R}$ be an envelope for $\mathcal{F}$, that is, $\sup_{f\in \mathcal{F}}|f(x)|\leq F(x)$ for each $x\in S$. We call $\mathcal{F}$ manageable for the envelope $F$ if there exists a decreasing function $D:(0,1]\rightarrow\mathbb{R}$ with $\int_{0}^1\sqrt{\log D(x)}\text{d} x < \infty$ such that, for every finite measure $Q$ on $S$ with finite support,
\[
D(\varepsilon\norm{F}_{L_2(Q)}, \mathcal{F},Q)\leq D(\varepsilon),\qquad \text{for } 0<\varepsilon\leq 1.
\]
The next theorem can be found in more general form as Theorem 4.4 in \cite{P1989}.

\begin{thm}\label{thm:pollard}
Let $\mathcal{F}$ be a manageable class for an envelope $F$ with $\ee{\abs{F(E)}^2}<\infty$, and for each $n\in\mathbb{N}$ let $\mathcal{F}(n)\subset\mathcal{F}$ be subclasses with $0 \in \mathcal{F}(n)$ such that $\sup_{f \in \mathcal{F}(n)} \ee{\abs{f(E)}} \to 0$ as $n\to\infty$. Then
\[
\ee{\sup_{f \in \mathcal{F}(n)} \Bigl|\sqrt{n}\Big(\frac{1}{n}\sum_{i=1}^{n}f(E_i)-\ee{f(E)}\Big)\Bigr|^2}\xrightarrow[n\to\infty]{} 0.
\]
\end{thm}

We want to use this to prove that
\[
\bar{E}_n^{q}R_n=\sqrt{n}\left(\frac{1}{n}\sum_{i=1}^n \Bigl(\abs{E_i-\bar{E}_n}^q-\abs{E_i-1}^q\Bigr)-\mathbb{E}\Bigl[\abs{E-\bar{E}_n}^q-\abs{E-1}^q\Bigr]\right)\xrightarrow[n\to\infty]{\mathbb{P}} 0.
\]
Let $(\delta_n)_{n\in\mathbb{N}}$ be a positive sequence satisfying $\delta_1=1$ and $\delta_n\to 0$, to be choosen later. For each $n\in\mathbb{N}$, define a class of functions on $\mathbb{R}$ by $\mathcal{F}(n):=\{x\mapsto f(x,t): t\in [1-\delta_n,1+\delta_n]\}$ with $f(x,t):=\abs{x-t}^q-\abs{x-1}^q$. Set $\mathcal{F}=\mathcal{F}_1$. We need the following Lipschitz-type inequality to prove that $\mathcal{F}$ is manageable. We will not present its proof which is basically an application of the mean value theorem.

\begin{lemma} \label{lem:qlipschitz}
Let $1\leq q<\infty$. For any $x,t_1,t_2\in\mathbb{R}$ with $\abs{t_1},\abs{t_2}\leq 2$
\[
\Bigl|\abs{x-t_1}^q-\abs{x-t_2}^q\Bigr|\leq q2^q\bigl(\abs{x}^{q-1}+2^{q-1}\bigr)|t_1-t_2|.
\]
\end{lemma}

\begin{lemma}\label{lem:manageability}
The class $\mathcal{F}$ is manageable with envelope $F(x):=q2^q \bigl(|x|^{q-1}+2^{q-1}\bigr)$.
\end{lemma}
\begin{proof}
With Lemma \ref{lem:qlipschitz} it is immediate that $F$, as defined above, is an envelope for $\mathcal{F}$ and that, for any two functions $f(\cdot,t_1)$ and $f(\cdot,t_2)$ in $\mathcal{F}$,
\[
\abs{f(x,t_1)-f(x,t_2)}\leq F(x)\abs{t_1-t_2}\quad \text{for all }x\in\mathbb{R}.
\]
By partitioning $[0,2]$ into disjoint subintervals and using a pidgeonhole argument, one can prove, similarly to \cite[Example 19.7]{V1998}, that $\mathcal{F}$ is manageable. We skip the details.
\end{proof}

Now that we have proven that $\mathcal{F}$ is manageable for the envelope $F$, we check the conditions of Theorem \ref{thm:pollard} and readily verify that $0\in \mathcal{F}(n)$. Also $\ee{\abs{F(E)}^2}<\infty$ since $E$ has finite moments of all orders. Finally, Lemma \ref{lem:qlipschitz} gives
\[
\sup_{f \in \mathcal{F}(n)} \ee{\abs{f(E)}}\leq \ee{F(E)}\delta_n\to 0.
\]
By Theorem \ref{thm:pollard} applied to $\mathcal{F}$ as above, Markov's inequality, and the fact that $\pr{\abs{\bar{E}_n-1}>\delta_n}\to 0$, as implied by the central limit theorem if we choose $\delta_n\geq n^{-1/2}\log n$, it follows that $R_n$ converges to zero in probability. With this, the proof of Corollary \ref{cor:cltqnorm} is complete.

As an example for the generality of the method, using the same arguments one can show the following result.
\begin{proposition}
Let $1\leq q<\infty$. Let $X,X_1,X_2,\ldots $ be i.i.d. continuous real-valued random variables with finite moments of order $2q$ such that $\sigma_q^2=\Var\bigl(\ee{\vert X-t\vert^q}'(\mu)X+\vert X-\mu\vert^q\bigr)>0$.
Then
\[
\sqrt{n}\left(\frac{1}{n}\sum_{i=1}^n \vert X_i - \bar{X}_n\vert^q - \ee{\vert X - \mu\vert^q}\right)\limd N\sim N(0,\sigma_q^2).
\]
\end{proposition}

\subsection*{Acknowledgement}
ZK has been supported by the German Research Foundation under Germany’s Excellence Strategy EXC 2044 – 390685587, Mathematics M\"unster: Dynamics - Geometry - Structure. JP and MS are supported by the Austrian Science Fund (FWF) Project P32405 \textit{Asymptotic geometric analysis and applications}. Most of this work was done while AB was a visiting PhD student at the University of Graz and we thank the RTG 2131 \textit{High-dimensional phenomena in probability - Fluctuations and discontinuity} for the financial support and the department for providing an optimal working environment.

\bibliographystyle{plain}
\bibliography{simplex}

\bigskip
	
	\bigskip
	
	\medskip
	
	\small
	
	\noindent \textsc{Anastas Baci:} Faculty of Mathematics,
	University of Bochum, 44780 Bochum, Germany
	
	\noindent {\it E-mail:} \texttt{anastas.baci@rub.de}
	
		\medskip
		
	\noindent \textsc{Zakhar Kabluchko:} Faculty of Mathematics, University of M\"unster,
		48149 M\"unster, Germany
		
	\noindent
		{\it E-mail:} \texttt{zakhar.kabluchko@uni-muenster.de}
	
		\medskip
	
	\noindent \textsc{Joscha Prochno:} Institute of Mathematics and Scientific Computing,
	University of Graz, 8010 Graz, Austria
	
	\noindent
	{\it E-mail:} \texttt{joscha.prochno@uni-graz.at}

	    \medskip
		
	\noindent \textsc{Mathias Sonnleitner:} Institute of Mathematics and Scientific Computing, University of Graz, 8010 Graz, Austria
		
	\noindent
		{\it E-mail:} \texttt{mathias.sonnleitner@uni-graz.at}
		
		\medskip
		
	\noindent \textsc{Christoph Th\"ale:} Faculty of Mathematics,
			University of Bochum, 44780 Bochum, Germany
			
	\noindent {\it E-mail:} \texttt{christoph.thaele@rub.de}

\end{document}